\documentclass[11pt]{article}
\usepackage{amsmath}
\usepackage{amsfonts}
\usepackage{amsthm}
\usepackage{amssymb, setspace}
\usepackage{color,graphicx, epsfig, geometry, hyperref,fancyhdr}
\usepackage{float}
\newtheorem{theorem}{Theorem}[section]

\newtheorem{corollary}{Corollary}[section]

\newtheorem{assumption}{Assumption}[section]
\newtheorem{remark}{Remark}[section]

\newcommand{\Z}{\mathbb{Z}}
\newcommand{\weakc}{\rightharpoonup}
\newcommand{\R}{\mathbb{R}}
\newcommand{\C}{\mathbb{C}}

\newcommand{\eps}{\epsilon}

\newcommand{\grad}{\nabla}

\newcommand{\ep}{\varepsilon}

\graphicspath{{paper-pics/}}

\begin{document}
\setlength{\parskip}{1mm}
\setlength{\oddsidemargin}{0.1in}
\setlength{\evensidemargin}{0.1in}
\lhead{}
\rhead{}

\begin{flushleft}
\Large 
\noindent{\bf \Large A direct method for reconstructing inclusions and boundary conditions from electrostatic data}
\end{flushleft}

\vspace{0.2in}

{\bf  \large Isaac Harris}\\
\indent {\small Department of Mathematics, Purdue University, West Lafayette, IN 47907, USA}\\
\indent {\small Email: \texttt{harri814@purdue.edu}}\\

%{\bf  \large William Rundell}\\
%\indent {\small Department of Mathematics, Texas A$\&$M University, College Station, TX 77843, USA} \\
%\indent {\small Email: \texttt{rundell@math.tamu.edu}}
%\vspace{0.2in}

%%%%%%%%%%%%%%%%%%%%%%%%%%%%%%%%%%%%%%%%%%%%

%%%%%%%%%%%%%%%%%%%%%%%%%%%%%%%%%%%%%%%%%%%%%%%%%%%%%%%%%%%%%%%%%%%%%%%%
\section{Introduction}\label{intro}
In this paper, we use direct methods (otherwise known as non-iterative methods) to reconstruct impenetrable inclusions from electrostatic Cauchy data. This problem models the non-destructive testing for interior inclusion using the voltage and current measurements on the accessible outer boundary. In particular, for a Dirichlet or Impedance inclusion we derive a sampling algorithm to recover the inclusion from the knowledge of the Dirichlet-to-Neumann (DtN) mapping. We focus on the case of Laplace's equation but the techniques used in this paper still hold for the case where the Laplacian is replaced with a uniformly elliptic operator in divergence form with sufficiently smooth coefficients. This gives a computationally simple way to solve the inverse problem of reconstructing the inclusion from the knowledge of DtN mapping. An important feature of sampling methods is that one does not need a prior information about the type or number of inclusions, unlike iterative methods where one needs to have some a prior knowledge of the inclusions to ensure convergence.  See \cite{iterative-inclusion, uniqueness, EIT-impedance} for examples of iterative methods applied to reconstructing impenetrable inclusions. Sampling algorithms have grown in popularity over the past two decades since their inception in \cite{CK} as a computationally simple way to recover obstacles. These method where first used to recover unknown obstacles from time-harmonic scattering data (see monographs \cite{CCbook,kirschbook} and the references therein). Over the years these methods have been employed to solve similar problems in the time domain. In \cite{TDLSM-wave,TDLSM-elastic} the Linear Sampling Method is applied to the acoustic and elastic wave equation, respectively. Recently in \cite{impedance-heat} the Linear Sampling Method was applied to recovering an impedance inclusion in a heat conductor.

Once the boundary of the inclusion is reconstructed, we then consider the problem of determining the boundary conditions on the interior boundary from the knowledge of the boundary and the DtN mapping. This amounts to solving our inverse problem in two steps where we first determine the boundary from the DtN mapping and then use the reconstructed boundary to determine the boundary conditions. Since we know that the Cauchy data on the outer boundary uniquely determines the electrostatic potential by the unique continuation principle we derive a system of boundary integral equations to reconstruct the Cauchy data on the interior boundary. From this one can determine the boundary condition on the interior boundary. We focus on the case of an impedance condition, where we provide an inversion method for determining the impedance parameter from the recovered Cauchy data. In our investigation of this problem we are able to show that the DtN mapping uniquely determines the $L^\infty$ impedance parameter. It should be noted that uniqueness for both the inclusion and the impedance condition follows from two pairs of Cauchy data (suitable chosen) from \cite{unique-imp} in the case of an inclusion with $C^{2, \alpha}$ boundary and $C^{1 , \alpha}$ impedance function.

The rest of the paper is structured as follows. We begin by formulating the direct and inverse problem under consideration. Next, we consider the problem of reconstructing the interior Dirichlet or Impedance boundary from the electrostatic Cauchy data. To this end, a sampling method is derived to determine the inclusion. We then turn our attention to reconstructing the impedance parameter given the interior boundary and the DtN mapping. Uniqueness is proven and a inversion algorithm is described using boundary integral equation. Lastly, we provide numerical experiments in 2-dimensions to show the feasibility of our inversion algorithm.  

%%%%%%%%%%%%%%%%%%%%%%%%%%%%%%%%%%%%%%%%%%%%%%%%%%%%%%%%%%%%%%%%%%%%%%%%
\section{Statement of the Direct and Inverse Problem}
We begin by considering the boundary value problems associated with the electrostatic problem with and without an impenetrable inclusions as derived from the quasi-static Maxwell's equations. Assume that $D \subset \R^d$ (for $d=2$ or 3) is a simply connected open set with $C^2$-boundary $\partial D$ with unit outward normal $\nu$. 
Now let $\Omega \subset D$ be a (possibly multiple) simply connected open set with $C^2$-boundary $\partial \Omega$, where we assume that $\text{dist}(\partial D, \partial \Omega)>0.$ 
For a material without an inclusion, we define $u \in H^1(D)$ to be the unique solution to the following boundary value problem   
\begin{align}
\Delta u=0 \quad \text{in} \quad  D \quad \text{with} \quad u  \big|_{\partial D}= f. \label{healthy}
\end{align} 
for a given $f \in H^{1/2}(\partial D)$. The function $u$ is the electrostatic potential for material without defects. Now for the defective material with an impenetrable  inclusion, we define $u_0 \in H^1(D \setminus \overline{\Omega})$ as the solution to 
\begin{align}
\Delta u_0=0 \quad \text{in} \quad  D \setminus \overline{\Omega}  \quad \text{with} \quad u_0  \big|_{\partial D}= f \quad \text{and} \quad \mathcal{B}(u_0) \big|_{\partial \Omega}=0. \label{defective}
\end{align} 
for a given $f \in H^{1/2}(\partial D)$. Here the function $u_0$ is the electrostatic potential for the defective material and the boundary operator $\mathcal{B}$ is given by 
$\mathcal{B}(u_0) = u_0$ the Dirichlet boundary condition on $\partial \Omega$ or  $\mathcal{B}(u_0) = \nu \cdot \grad u_0 + \gamma(x) u_0$  the impedance boundary condition on $\partial \Omega$.
We assume that the impedance parameter $\gamma(x)$ is a non-trivial function in 
$$L^{\infty }_{+}(\partial \Omega) := \Big\{ \gamma \in L^{\infty}(\partial \Omega) \, \, : \, \, \inf\limits_{\partial \Omega} \gamma (x) \geq 0 \Big\}.$$
Here we take $\nu$ to be the unit outward normal to the domain $D \setminus \overline{\Omega}$, see Figure \ref{dp-pic}. 
\begin{figure}[H]
\centering
\includegraphics[scale=0.3]{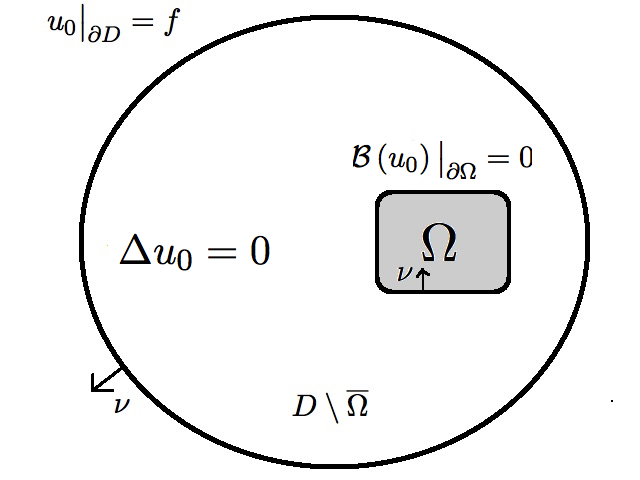}
\caption{ The electrostatic problem for material with an inclusion. }
\label{dp-pic} 
\end{figure}

Assume that the `voltage' $f$ is applied on the boundary $\partial D$ and the current $ \nu \cdot \grad u_0 = \partial_\nu u_0$  is `measured' on $\partial D$. From these measurements we wish to reconstruct the impenetrable inclusion $\Omega$ without any a prior knowledge of the number of inclusions or boundary condition on $\partial \Omega$. Also, if we assume that it is known {\it a priori} that the boundary condition is of impedance type we wish to also reconstruct the parameter $\gamma$.

For the case of a perfectly conducting inclusion i.e. zero Dirichlet condition on $\partial \Omega$ it has been shown in \cite{uniqueness} that a single pair of voltage and current measurements on $\partial D$ can be used to determine $\partial \Omega$. In \cite{iterative-inclusion,uniqueness}  iterative methods based on conformal mapping is used to reconstruct a single perfectly conducting inclusion. The question of unique determination of the boundary $\partial \Omega$ and impedance condition $\gamma (x)$ is more involved than for the case of the perfectly conduction inclusion. It has been shown in \cite{unique-imp} that two pairs of voltage and current measurements on the boundary $\partial D$ are enough to determine $\partial \Omega$ and $\gamma(x)$ provided the currents are linearly independent and non-negative assuming that $\gamma (x) \in C^{1,\alpha}(\partial \Omega)$ and $\partial \Omega$ is class $C^{2,\alpha}$ for $0<\alpha <1$. See \cite{C-Map-imp,EIT-impedance} where iterative methods are proposed to determine the inclusion and the impedance. The goal here is to develop a sampling method to reconstruct the boundary of the inclusion $\partial \Omega$ then one can reconstruct $\gamma(x)$ using a system of boundary integral equations, which is less computationally expensive to reconstruct the impedance parameter once the boundary is known.

By our assumptions on the impedance parameter $\gamma(x)$ it can easily  be shown that both \eqref{healthy} and \eqref{defective} are well-posed using variational techniques (see Chapter 5 of \cite{evans}) for the Dirichlet or impedance boundary condition on the inclusion. By the linearity of the equation and boundary conditions we have that the voltage to electrostatic potential mappings 
$$ f \longmapsto u  \quad \text{and} \quad f \longmapsto u_0 $$
are bounded linear operators from $H^{1/2}(\partial D)$ to $H^1(D)$ and $H^1 \big(D \setminus \overline{\Omega} \big)$, respectively. We now define the DtN mappings from $H^{1/2}(\partial D) \longmapsto H^{-1/2}(\partial D)$ such that 
$$\Lambda f=\partial_{\nu} u \big|_{\partial D} \quad \text{ and } \quad \Lambda_0 f=\partial_{\nu} u_0 \big|_{\partial D}.$$
Due to the well-posedness of \eqref{healthy} and \eqref{defective} along with the Trace Theorem (see for e.g. \cite{evans}) it follows that the DtN mappings are bounded linear operators. The {\it inverse shape problem} we consider here is to reconstruct the support of the impenetrable inclusion $\Omega$ from the knowledge of the DtN mappings $\Lambda$ and $\Lambda_0$, i.e. we want to determine the boundary $\partial \Omega$ from the set of all possible measurements $(f \, , \partial_{\nu} u)$ and $(f \, , \partial_{\nu} u_{0})$ on $\partial D$. Moreover, for the case where $\mathcal{B}(u_0)$ is given by the impedance boundary condition on $\partial \Omega$ we consider the {\it inverse impedance problem} of recovering the impedance function $\gamma$ from the knowledge of the DtN mappings $\Lambda_0(\gamma)$.

%%%%%%%%%%%%%%%%%%%%%%%%%%%%%%%%%%%%%%%%%%%%%%%%%%%%%%%%%%%%%%%%%%%%%%%%
\section{A Sampling Method for the Inverse Shape Problem}
We now derive a Sampling Method for our inverse {shape} problem. Therefore, we now consider the inverse problem of reconstructing $\partial \Omega$ from the knowledge of $\Lambda_0 = \Lambda_0(\partial \Omega)$. The goal is to first reconstruct the boundary $\partial \Omega$ via a sampling method and provided that $\mathcal{B}(u_0)$ is given by the impedance boundary condition we then reconstruct the impedance parameter using boundary integral equations in the following section. The sampling method is based on connecting the support of the unknown region to an ill-posed equation involving the operator defined by the measurements (i.e. the difference of the DtN mappings).  First we decompose the difference of the DtN mappings and analyze the operators used in our decomposition, then we develop the sampling methods using the decomposition to derive an inversion algorithms to determine the support of the inclusion. For the case of multiple inclusions we let $\Omega= \bigcup_{m=1}^{M} \Omega_m$ where $\Omega_m \subset D$ is a simply connected open set with $C^2$-boundary $\partial \Omega_m$ such that $\Omega_i \cap \Omega_j$ is empty for all $i \neq j$. All of the analysis in the preceding sections holds where the space Trace spaces $H^{\pm 1/2}(\partial \Omega)$ are understood as the product space $H^{\pm 1/2}(\partial \Omega_1) \times \cdots \times H^{\pm 1/2}(\partial \Omega_M)$. 

\begin{remark} \label{by-the-way}
The results in section can easily be extended for the case when the Laplacian is replaced by $\grad \cdot A(x) \grad$ where $A(x) \in C^1(D, \C^{d \times d})$ is given by $$ A(x)=\sigma(x) - \text{i} \, \omega \eps(x)$$
where the electric conductivity $\sigma(x)$ and electric permittivity $\eps(x)$ are symmetric real valued matrices such that
$$ \overline{\xi} \cdot  \sigma(x)  \xi \geq \sigma_{\text{min} } |\xi|^2 \quad \text{and} \quad  \overline{\xi} \cdot  \eps(x)  \xi \geq  0$$
for all $\xi\in{\mathbb C}^d$ for almost every $x\in D$ with frequency $\omega \geq 0$. 
\end{remark}

We now develop a sampling method to reconstruct the support of $\Omega$ from a knowledge of the difference of the DtN operators $(\Lambda - \Lambda_0)$ for the Dirichlet or Impedance boundary condition. As we will see later in this section this method can be used without having any a prior information about the type or number of inclusions. To do so, we will decompose the difference of the DtN operators using two operators, the first mapping the voltage $f$ to an appropriate boundary value on $\partial \Omega$ and the second mapping takes the aforementioned boundary values on $\partial \Omega$ to the difference of the currents $\partial_{\nu} (u-u_0)$ on $\partial D$.

Notice, that the difference of the currents on $\partial D$ is given by $(\Lambda - \Lambda_0)f$ on $\partial D$. Therefore, consider the difference of the solutions $u-u_0$ in $H^1( D \setminus \overline{\Omega})$ which satisfies 
\begin{subequations}
\begin{align}
\Delta  (u-u_0 )=0 \quad &\text{in} \quad  D \setminus \overline{\Omega} \label{difference1}  \\ 
(u-u_0)  \big|_{\partial D}= 0 \quad &\text{and} \quad  \partial_{\nu} (u-u_0 ) \big|_{\partial \Omega} \in H^{-1/2}(\partial \Omega). \label{difference}
\end{align} 
\end{subequations}
Now, motivated by equation \eqref{difference1}-\eqref{difference} we let $w \in H^1( D \setminus \overline{\Omega})$ be the unique solution to 
\begin{align}
\Delta  w=0 \quad &\text{in} \quad  D \setminus \overline{\Omega}   \quad \text{with} \quad w \big|_{\partial D}= 0 \quad \text{and} \quad  \partial_{\nu} w \big|_{\partial \Omega} =h \label{equ-w}
\end{align} 
for a given $h \in H^{-1/2}(\partial \Omega)$. Again, using a variational method one can show that equation \eqref{equ-w} is well-posed, so we can define via the Trace Theorem the bounded linear operator
$$ G: H^{-1/2}(\partial \Omega) \longmapsto H^{-1/2}(\partial D) \quad \text{ given by } \quad Gh = \partial_{\nu} w \big|_{\partial D}$$
where $w$ is the unique solution to equation \eqref{equ-w}. Notice that for $h= \partial_{\nu} (u-u_0 ) \big|_{\partial \Omega}$ then $Gh=(\Lambda - \Lambda_0)f$. 
We now define the bounded linear operators $H^{1/2}(\partial D) \longmapsto H^{-1/2}(\partial \Omega)$
such that 
$$  L f =  \partial_{\nu} u\big|_{\partial \Omega} \quad \text{ and } \quad L_0 f =  \partial_{\nu} u_0 \big|_{\partial \Omega}$$
where $u$ and $u_0$ are the solutions of \eqref{healthy} and \eqref{defective} respectively. This gives the following decomposition.

\begin{theorem}\label{decomp}
The difference of the DtN operators $$(\Lambda - \Lambda_0)  : H^{1/2}(\partial D) \longmapsto H^{-1/2}(\partial D)$$ associated with \eqref{healthy} and \eqref{defective} has the factorization $ (\Lambda - \Lambda_0) = G(L-L_0)$.
\end{theorem}

We now analyze the operators used to factorize the difference of the DtN operators. We begin by analyzing the operator
$$G: H^{-1/2}(\partial \Omega) \longmapsto H^{-1/2}(\partial D).$$ 
In the following results we obtain the properties of this operator. 

\begin{theorem}\label{G}
The operator $G: H^{-1/2}(\partial \Omega) \longmapsto H^{-1/2}(\partial D)$ {given by} $Gh = \partial_{\nu} w \big|_{\partial D}$
where $w$ is the unique solution to equation \eqref{equ-w}  is compact and injective. 
\end{theorem}
\begin{proof}
We begin by proving the compactness. Notice that by interior elliptic regularity (see for e.g. \cite{evans}) we have that for any $h \in H^{-1/2}(\partial \Omega)$  the solution to \eqref{equ-w} is in $H^2_{loc}(D \setminus \overline{\Omega}).$ Now let $D_0$ be such that $\Omega \subset \overline{D}_0$ and $ \overline{D}_0 \subset D$ where $\partial D_0$ is class $C^2$ and let $w$ be the solution to \eqref{equ-w} which gives that the trace of $w$ on $\partial D_0$ is in $H^{3/2}(\partial D_0)$. This implies that $w$ satisfies 
$$\Delta  w=0 \quad \text{in} \quad  D \setminus \overline{D}_0 \quad \text{with} \quad  w \big|_{\partial D}= 0 \quad \text{and} \quad  w \big|_{\partial D_0} \in  H^{3/2}(\partial D_0)$$
and by global elliptic regularity \cite{evans} we have that $w \in H^2(D \setminus \overline{D}_0)$. The Trace Theorem implies that  $Gh=\partial_{\nu} w \big|_{\partial D} \in H^{1/2}(\partial D)$ and the compact embedding of $H^{1/2}(\partial D)$ into $H^{-1/2}(\partial D)$ proves the claim. 

Now we prove the injectivity. Let $h \in \text{Null} (G)$ and $w$ be the solution to \eqref{equ-w} with boundary data $h$. Therefore, we have that 
$$\Delta w=0 \quad \text{in} \quad  D \setminus \overline{\Omega} \quad  \text{and} \quad  w = \partial_{\nu} w =0 \quad \text{on} \quad  \partial D.$$
By appealing to unique continuation we can conclude that $w=0$ in $D \setminus \overline{\Omega}$ and therefore the Trace Theorem gives that $h=0$, proving the claim. 
\end{proof}

To analyze the operator $G$ further we now compute it's Transpose (Dual) operator $G^{\top}$.
Therefore, let $\langle \cdot \, ,\cdot  \rangle_{\Gamma}$ denote the dual pairing between $H^{1/2}(\Gamma)$ and $H^{-1/2}(\Gamma)$ (with $L^2(\Gamma)$ as the pivot space) and by definition 
$$\langle \, G^{\top} \varphi , h  \, \rangle_{\partial \Omega} = \langle \, \varphi , G h   \, \rangle_{\partial D} = \int\limits_{\partial D} {\varphi} \,   \partial_{\nu} w  \, \text{d}s \quad \text{for all} \quad \varphi \in H^{1/2}(\partial D)\, \, \, \text{and} \, \, \, h \in H^{-1/2}(\partial \Omega). $$
Now take a  lifting of the function $\varphi \in H^{1/2}(\partial D)$ such that $v \in H^1\big(D \setminus \overline{\Omega}\, \big)$ is the unique solution to 
\begin{align}
\Delta  v=0 \quad &\text{in} \quad  D \setminus \overline{\Omega}  \quad \text{ with } \quad v \big|_{\partial D}= \varphi \quad \text{and} \quad  \partial_{\nu} v \big|_{\partial \Omega} =0. \label{equ-v}
\end{align} 
Applying Green's 2nd Theorem and using the boundary value problems \eqref{equ-w} and \eqref{equ-v} gives  
$$\langle  G^{\top} \varphi , \, h \, \rangle_{\partial \Omega} =\int\limits_{\partial D} {\varphi} \,   \partial_{\nu} w  \, \text{d}s =  - \int\limits_{\partial \Omega} {v} \,  h  \, \text{d}s $$
and we can conclude that 
$$G^{\top} : H^{1/2}(\partial D) \longmapsto H^{1/2}(\partial \Omega) \quad  \text{is given by} \quad G^{\top} \varphi = - v \big|_{\partial \Omega}$$
where $v$ is the unique solution to equation \eqref{equ-v}. 

Just as in the proof of Theorem \ref{G} one can clearly see that due to the unique continuation principal that $G^{\top}$ is injective and therefore since (see for e.g. \cite{McLean})
$$ \overline{ \text{Range}(G) } = ^{a} \hspace{-0.025in}\text{Null}\big(G^{\top} \big)$$
(here $^a$ denotes the annihilators) we have the following result. 

\begin{theorem}
The operator $G: H^{-1/2}(\partial \Omega) \longmapsto H^{-1/2}(\partial D)$ {given by} $Gh = \partial_{\nu} w \big|_{\partial D}$
where $w$ is the unique solution to equation \eqref{equ-w} has a dense range. 
\end{theorem}

Now we turn our attention to the injectivity of the operator $(L- L_0 )$. 

\begin{assumption}\label{assume}
Assume that for any $g \in H^{-1/2}(\partial \Omega)$ that 
$$\Delta \phi = 0 \quad \text{in} \quad \Omega \ \quad \text{and} \quad  \partial_{\nu} \phi + \gamma \phi = g \quad \text{on} \quad \partial \Omega$$
has a unique solution $\phi \in H^1(\Omega)$ depending continuously on the boundary data. Here $\nu$ is the unit inward norm to the boundary $\partial \Omega$. 
\end{assumption}  
Since $\nu$ is the inward pointing normal it is clear that uniqueness is not guaranteed since the boundary condition will have the wrong sign for the positive impedance parameter. Note that Assumption \ref{assume} is a common feature of sampling method. In \cite{lsmaniso} where the linear sampling method is used to reconstruct anisotropic obstacles using time-harmonic acoustic measurements one must assume that the corresponding wave number is not a so-called interior transmission eigenvalue of the obstacle. 

In our case, Assumption \ref{assume}  says that $\lambda=1$ is not an associated weighted Steklov eigenvalue given by the values $\lambda \in \R$ such that there is a nontrivial solution to 
$$\Delta \phi = 0 \quad \text{in} \quad \Omega \ \quad \text{and} \quad  \partial_{\nu} \phi + \lambda \, \gamma(x) \phi =0 \quad \text{on} \quad \partial \Omega.$$
Since the set of eigenvalues is discrete $\lambda=1$ is almost surely not an eigenvalue for a given domain $\Omega$ and impedance $\gamma(x)$. With Assumption \ref{assume} we now consider the injectivity of of the operator $(L- L_0 )$.

\begin{theorem}
The operator 
$$(L- L_0 ): H^{1/2}(\partial D) \longmapsto H^{-1/2}(\partial \Omega) \quad  \text{given by} \quad (L - L_0) f =  \partial_{\nu}(u- u_0) \big|_{\partial \Omega}$$
where $u$ and $u_0$ are the solutions of \eqref{healthy} and \eqref{defective} is injective.
\end{theorem}
\begin{proof}
To prove the injectivity we split the proof into two parts for the two boundary conditions on $\partial \Omega$ under consideration. First assume that $\mathcal{B}$ is the Dirichlet boundary condition on $\partial \Omega$ and let $f \in \text{Null} (L-L_0)$, therefore by definition we have that $\partial_{\nu} (u-u_0) =0$ on $ {\partial \Omega}$ where $u$ and $u_0$ are the solutions of \eqref{healthy} and \eqref{defective} respectively. This implies that the difference $u-u_0$ solve \eqref{equ-w} with boundary data $h=0$ and by well-posedness we conclude that $u=u_0$ in $D \setminus \overline{\Omega}$. We now have  
$$\Delta  u=0 \quad \text{in} \quad  \Omega \quad  \text{and} \quad  u = 0 \quad \text{on} \quad  \partial \Omega$$
which it follows that $u=0$ in $\Omega$. By unique continuation we have $u=0$ in $D$ which gives that $f=0$. 

Similarly for the Impedance boundary condition on $\partial \Omega$ if we let $f \in \text{Null} (L-L_0)$ then we can conclude that $u=u_0$ in $D \setminus \overline{\Omega}$. This implies 
$$ \Delta  u=0 \quad \text{in} \quad  \Omega \quad  \text{and} \quad   \partial_{\nu} u + \gamma(x) u = 0 \quad \text{on} \quad  \partial \Omega$$
which implies that $u=0$ in $\Omega$ by Assumption \ref{assume}. By again appealing to unique continuation and we conclude that $f=0$, proving the claim. 
\end{proof}

Recall, that the difference of the DtN operators has the decomposition $(\Lambda - \Lambda_0) = G(L-L_0)$. Since $L$ and $L_0$ are both bounded linear operators by appealing to the previous results we have the following.

\begin{corollary}\label{compact/injective}
The difference of the DtN operators $$(\Lambda - \Lambda_0)  : H^{1/2}(\partial D) \longmapsto H^{-1/2}(\partial D)$$ is compact and injective.
\end{corollary}

We now derive a sampling method to solve our inverse problem. Sampling methods often connect the support of the region of interest to an ill-posed problem where one uses a singular solution to the background equation. The idea is to show that due to the singularity that a particular equation is ``not" solvable unless the singularity is contained in the region of interest. To this end, we prove the following results to derive our inversion method.

\begin{theorem}\label{dense-range}
The difference of the DtN operators  
$$(\Lambda - \Lambda_0)  : H^{1/2}(\partial D) \longmapsto H^{-1/2}(\partial D)$$
is symmetric (i.e. is equal to it's transpose) and therefore has a dense range.
\end{theorem}
\begin{proof}
To begin, we let $f_j \in H^{1/2}(\partial D)$ where $u^{(j)} \in H^1(D)$ and $u_0^{(j)} \in H^1(D\setminus \overline{\Omega})$ are the unique solutions to \eqref{healthy} and \eqref{defective}, respectively for $j=1,2$. We now consider 
$$\big\langle f_1 , (\Lambda - \Lambda_0) f_2 \big\rangle_{\partial D}$$
where again $\langle \cdot \, ,\cdot  \rangle_{ \partial D}$ denotes the dual pairing between $H^{1/2}(\partial D)$ and $H^{-1/2}(\partial D)$. By definition we have that 
\begin{align*}
\big\langle f_1 , (\Lambda - \Lambda_0) f_2 \big\rangle_{\partial D}&= \int\limits_{\partial D} f_1\,  \partial_{\nu}u^{(2)} -  f_1 \, \partial_{\nu}u^{(2)}_0 \, \text{d}s= \int\limits_{\partial D} u^{(1)} \partial_{\nu} u^{(2)} \, ds -\int\limits_{\partial D} u_0^{(1)} \partial_{\nu} u_0^{(2)} \, \text{d}s 
\end{align*}
Now, by Green's 1st Theorem 
\begin{align*}
\big\langle f_1 , (\Lambda - \Lambda_0) f_2 \big\rangle_{\partial D}= \int\limits_{ D} \grad  u^{(1)} \cdot  \grad  u^{(2)} \, \text{d}x - \int\limits_{ D \setminus \overline{\Omega} } \grad u_0^{(1)} \cdot  \grad  u_0^{(2)} \, \text{d}x + \int\limits_{\partial \Omega } u_0^{(1)} \partial_{\nu} u_0^{(2)} \, \text{d}s.
\end{align*}
For the Dirichlet boundary condition on $\partial \Omega$ we obtain 
$$\big\langle f_1 , (\Lambda - \Lambda_0) f_2 \big\rangle_{\partial D}= \int\limits_{ D} \grad  u^{(1)} \cdot  \grad  u^{(2)} \, \text{d}x - \int\limits_{ D \setminus \overline{\Omega} } \grad u_0^{(1)} \cdot \grad  u_0^{(2)} \, \text{d}x$$
and for the Impedance boundary condition conclude that 
$$\big\langle f_1 , (\Lambda - \Lambda_0) f_2 \big\rangle_{\partial D}= \int\limits_{ D} \grad  u^{(1)} \cdot  \grad  u^{(2)} \, \text{d}x - \int\limits_{ D \setminus \overline{\Omega} } \grad u_0^{(1)} \cdot \grad  u_0^{(2)} \, \text{d}x - \int\limits_{\partial \Omega } \gamma(x) \, u_0^{(1)} u_0^{(2)} \, \text{d}s.$$
Therefore, we have that the right hand side of the above expressions are symmetric bilinear forms and therefore $(\Lambda - \Lambda_0)$ is symmetric. By Corollary  \ref{compact/injective} we can conclude that $(\Lambda - \Lambda_0)$ has a dense range. 
\end{proof}

We define $\mathbb{G} (x,z)$ as the Green's function for $D$ which is the solution to
$$ \Delta  \mathbb{G} (\cdot \, , \, z) =- \delta(\cdot - z)  \quad \text{in} \quad  D \quad  \text{and} \quad  \mathbb{G} (\cdot \, ,  z) =0 \quad \text{on} \quad  \partial D.$$
We now connect the support of the inclusion $\Omega$ to the range of the operator $G$. 

\begin{theorem}\label{range}
$\partial_{\nu} \mathbb{G} (\cdot \, ,  z) \big|_{\partial D} \in \text{Range}(G)$ if and only if $z \in \Omega$.  
\end{theorem}
\begin{proof}
Notice that for $z \in \Omega$, $\mathbb{G}( \cdot \, ,z) \in H^1(D \setminus \overline{\Omega})$ is harmonic in the annular region and satisfies \eqref{equ-w} with $h_z=\partial_{\nu} \mathbb{G} (\cdot \, ,  z)$ on $\partial \Omega$. It is clear that $Gh_z=\partial_{\nu} \mathbb{G} (\cdot \, ,  z) \big|_{\partial D}$. 

Now, assume that $\partial_{\nu} \mathbb{G} (\cdot \, ,  z) \big|_{\partial D} \in \text{Range}(G)$  for some $z \in D \setminus \overline{\Omega}$. Therefore, we can conclude that there is a $w_z$ solving \eqref{equ-w} such that 
$$\partial_{\nu} w_z=\partial_{\nu} \mathbb{G} (\cdot \, ,  z) \quad \text{on} \quad {\partial D}.$$
We now define $U_z=w_z-\mathbb{G}(\cdot \, , z)$ which satisfies 
$$\Delta  U_z =0 \quad \text{in} \quad D \setminus \big(\overline{\Omega} \cup \{z\} \big) \quad  U_z =\partial_{\nu} U_z=0 \quad \text{on} \quad \partial D.$$
Holmgren's Theorem implies that $w_z= \mathbb{G} (\cdot \, ,  z)$ in $D \setminus \big(\overline{\Omega} \cup \{z\} \big)$, but interior elliptic regularity gives that $w_z$ is bounded as $x \rightarrow z$ where as $|\mathbb{G} (x,z)| \rightarrow \infty$ as $x \rightarrow z$, proving the claim by contradiction. 
\end{proof}

Next we turn our attention to showing that the linear sampling method can be applied as an inversion method for our inverse problem. The linear  sampling method was first derived in \cite{CK} as a way to reconstruct impenetrable obstacles using time-harmonic acoustic waves. We now show that a sampling algorithm can be used to reconstruct the inclusion. 

From the analysis given in this section we now have all we need to derive sampling method for reconstructing $\Omega$. To this end, consider the ill-posed `current-gap' equation
\begin{equation}
(\Lambda - \Lambda_0) f_z = \phi_z  \quad \text{for} \quad z \in D \quad \text{where} \quad  \phi_z = \partial_{\nu} \mathbb{G} (\cdot \, ,  z) \big|_{\partial D}. \label{PGE} 
\end{equation}
By Theorem \ref{dense-range} we have that for all $z \in D$ we have that there exists an approximating sequence $\big\{f_{z , \ep} \big\}_{\ep >0}$ of solutions to \eqref{PGE} where 
$$\big\| (\Lambda - \Lambda_0) f_{z , \ep} - \phi_z  \big\|_{H^{-1/2}(\partial D)} \longrightarrow 0 \quad \text{as } \ep \rightarrow 0.$$
Now assume that $\| f_{z , \ep} \|_{H^{1/2}(\partial D)}$ is bounded as $\ep \rightarrow 0$. Since $f_{z , \ep} \in H^{1/2}(\partial D)$ is a bounded sequence we have that there is a weakly convergent subsequence (still denote with $\ep$) such that $f_{z , \ep} \weakc f_{z,0}$ as $\ep \to 0$. Since $(\Lambda - \Lambda_0)$ is compact we can conclude that 
$$ (\Lambda - \Lambda_0) f_{z , \ep}  \longrightarrow (\Lambda - \Lambda_0) f_{z,0} \quad \text{and} \quad (\Lambda - \Lambda_0) f_{z , \ep} \longrightarrow \phi_z  \quad \text{as} \quad \ep \rightarrow 0 \quad \text{ in} \,\,\,H^{-1/2}(\partial D).$$
By the decomposition given in Theorem \ref{decomp} this implies that $\phi_z \in$ Range$(G)$, which is a contradiction of Theorem \ref{range} if $z \notin \Omega$. 

From the above analysis we have derived a sampling method for recovering the unknown inclusion $\Omega$ by constructing approximate solutions to \eqref{PGE}.

\begin{theorem}\label{LSM}
Let $\phi_z = \partial_{\nu} \mathbb{G} (\cdot \, ,  z) \big|_{\partial D}$. Then for any sequence 
$\big\{f_{z , \ep} \big\}_{\ep >0} \in H^{1/2}(\partial D)$ that  is an approximating solution of \eqref{PGE} such that 
$$\big\| (\Lambda - \Lambda_0) f_{z , \ep} - \phi_z  \big\|_{H^{-1/2}(\partial D)} \longrightarrow 0 \quad \text{as } \ep \rightarrow 0$$
then $\| f_{z , \ep} \|_{H^{1/2}(\partial D)} \longrightarrow \infty$ as $\ep \rightarrow 0$ for all $z \notin \Omega$. 
\end{theorem}

Notice that Theorem \ref{LSM} says that equation \eqref{PGE} is not ``approximately solvable" provided that $z \notin \Omega$ i.e. there is no sequence of approximate solutions whose (weak) limit satisfies \eqref{PGE}. Since we assume that $(\Lambda - \Lambda_0)$ and $\phi_z$ are known we can use a regularization strategy to find an approximate  solution to the current-gap equation \eqref{PGE}. Also notice that it does not matter if one has the Dirichlet or impedance boundary condition on $\partial \Omega$, Theorem \ref{LSM} is valid for either case. One can easily modify the analysis in this section to show that Theorem \ref{LSM} is also valid for the perfectly insulated inclusion where $\mathcal{B}(u_0)=\partial_\nu u_0$. This gives that the sampling method is robust in the fact that it can be applied for multiple boundary conditions.  The inversion algorithm for reconstructing the boundary $\partial \Omega$ is as follows: choose a grid of points in $D$, for each grid point `solve' \eqref{PGE} via a regularization strategy, plot the $W(z)= \| f_{z , \ep} \|^{-1}_{H^{1/2}(\partial D)}$ where $ f_{z , \ep}$ is the regularized solution to \eqref{PGE} and set the reconstruction ${\partial \Omega_\delta}=\big\{ z \in D \, \, : \, \,  W(z) = \delta \ll1  \big\}$.

%%%%%%%%%%%%%%%%%%%%%%%%%%%%%%%%%%%%%%%%%%%%%%%%%%%%%%%%%%%%%%%%%%%%%%%%
\section{ Integral Equations for the Inverse Impedance Problem}\label{BIE}

In this section, we will derive a non-iterative method for reconstructing the impedance parameter $\gamma (x)$. Even though we focus on the case of the impedance boundary condition the reconstruction method presented in this section works for determining if the inclusion is perfectly conducting/insulated. To this end, we consider the inverse problem of reconstructing the boundary impedance from the knowledge of $\Lambda_0 (\gamma)$. We assume that the boundary $\partial \Omega$ is known and that the DtN mapping which maps $u_0 = f$ on $\partial D$ to $\partial_{\nu}  u_0 = g$ on $\partial D$ is given on some subset of $H^{1/2}(\partial D)$. The idea is to use the knowledge of the Cauchy data on $\partial D$ to recover the corresponding Cauchy data on $\partial \Omega$. Once we have the Cauchy data on $\partial \Omega$ of $u_0$ the impedance parameter can be determined by solving $\partial_{\nu}  u_0 +\gamma(x) u_0=0$ on $\partial \Omega$. 

We begin this section by proving a uniqueness for the inverse problem. Since we assume that the DtN mapping is known we wish to prove that the {\it inverse impedance} and {\it inverse shape} problems admits a unique solution. Since we assume that we have an infinite data set we should be able to prove uniqueness for sufficiently less regularity than is needed in \cite{unique-imp}. To do so, we first need the following Theorem. 

\begin{theorem}\label{dense-set}
The set $$\, {\mathcal U} := \Big\{ u_0 \big|_{\partial \Omega} \, \, : \, \,  u_0  \in H^1( D \setminus \overline{\Omega})  \, \, \text{  solving  } \eqref{defective} \, \,  \text{ for all } \,  \, f \in  H^{1/2}(\partial D) \Big\}$$
is a dense subspace of $L^2(\partial \Omega)$. 
\end{theorem}
\begin{proof}
To prove the claim let $\phi \in L^2(\partial \Omega)$ and assume that 
$$\int\limits_{\partial \Omega} u_0 \phi \, \text{d}s = 0 \quad \text{ for all } \quad f \in  H^{1/2}(\partial D).$$
Now let $v \in H^1(D \setminus \overline{\Omega})$ be the unique solution to 
$$ \Delta v=0 \quad \text{in} \quad D\setminus \overline{\Omega} \quad \text{ with } \quad v=0 \,\, \text{ on} \quad \partial D \quad \text{ and } \quad \partial_{\nu} v+ \gamma v= \phi \,\, \text{ on } \quad \partial \Omega.$$
Using Green's 2nd Theorem we have that 
\begin{align*}
0& = \int\limits_{D\setminus \overline{\Omega} }  u_0 \Delta v - v\Delta u_0 \, \text{d}x  = \int\limits_{\partial D} u_0 \partial_\nu v - v \partial_\nu u_0 \, \text{d}s + \int\limits_{\partial \Omega} u_0 \partial_\nu v - v \partial_\nu u_0 \, \text{d}s.
\end{align*}
Appealing to the boundary conditions for both $u_0$ and $v$ gives 
\begin{align*}
-\int\limits_{\partial D} f \partial_\nu v \, \text{d}s & = \int\limits_{\partial \Omega} u_0 ( \partial_\nu v + \gamma v) \, \text{d}s = \int\limits_{\partial \Omega} u_0 \phi \, ds = 0. 
\end{align*}
This implies that $\partial_\nu v=0$ on $\partial D$ since it is orthogonal to all $f \in  H^{1/2}(\partial D)$. Since $v$ has zero Cauchy data on $\partial D$ we have that $v=0$ in $D \setminus \overline{\Omega}$ and the Trace Theorem gives that $\phi = 0$. Proving the claim. 
\end{proof}

Notice that Theorem \ref{dense-set} hold true for any dense subset of $H^{1/2}(\partial D)$. We can now prove that the impedance is uniquely determined by the knowledge of the DtN mapping on any dense subset of $H^{1/2}(\partial D)$. 

\begin{theorem} \label{unique} 
The DtN mapping $\Lambda_0 : H^{1/2}(\partial D) \longmapsto H^{-1/2}(\partial D)$ uniquely determines  impedance parameter $\gamma (x) \in L^{\infty}_{+}(\partial \Omega)$. 
\end{theorem}
\begin{proof}
Assume that $\Lambda_{0 }(\gamma_1) = \Lambda_{0 } ( \gamma_2)$ then let $u_0^{(j)}$ be the solution to \eqref{defective} with impedance $\gamma_j$ for $j=1,2$. Therefore, we have that $u_0^{(1)}$ and $u_0^{(2)}$ have the same Cauchy data on $\partial D$ which implies that $u_0 = u_0^{(1)}=u_0^{(2)}$ in $D \setminus \overline{\Omega}$ by unique continuation. We can conclude that
$$ \partial_{\nu} u_0 + \gamma_1 u_0 =0 \quad \text{and} \quad  \partial_{\nu} u_0 + \gamma_2 u_0  =0 \quad \text{ on } \, \, \partial \Omega. $$
Subtracting the impedance conditions implies that $(\gamma_1 - \gamma_2) u_0 = 0$ on $\partial \Omega$ for all $f \in  H^{1/2}(\partial D)$. We conclude that $(\gamma_1 - \gamma_2)$ is orthogonal to the set $\mathcal{U}$ and is therefore zero a.e. on $\partial D$ proving the claim. 
\end{proof}

\begin{remark}
The above proof is carried out in a variational setting so the uniqueness holds for the case where that Laplacian is replaced with $\grad \cdot A(x) \grad$ where the symmetric coefficient matrix satisfies the same assumptions as in Remark \ref{by-the-way}. 
\end{remark}

We now turn our attention to deriving an inversion method for determining $\gamma (x)$ from the knowledge of the DtN mapping $\Lambda_0$ and $\partial \Omega$. Our inversion method requires us to write the electrostatic potential function $u_0$ in terms of boundary integral operators. To this end, we adopt the notation $\partial D= \Gamma_{\text{m}}$ (i.e. the measurements boundary) and $\partial \Omega= \Gamma_{\text{i}}$ (i.e. the impedance boundary). Therefore, since both boundaries are assumed to be $C^2$ we define 
$$\mathcal{D}_{\text{m}}: H^{1/2}(\Gamma_{\text{m}}) \mapsto H^1(D) \cup H^1_{loc}(\R^d \setminus \overline{D}) \quad \text{ and } \quad \widetilde{ \mathcal{D}}_{\text{i}}: H^{1/2}(\Gamma_{\text{i}}) \mapsto H^1(\Omega) \cup H^1_{loc}(\R^d \setminus \overline{\Omega})$$
by the boundary integral operators 
$$ (\mathcal{D}_{\text{m}}\,  \varphi)(x) = 2\int\limits_{\Gamma_{\text{m}} } \varphi (y) \partial_{\nu(y)} \Phi(x,y) \, \text{d}s_y \quad \text{ for } \, \, x \in \R^d \setminus \Gamma_{\text{m}}$$
and 
$$(\widetilde{\mathcal{D}}_{\text{i}} \, \psi)(x) = 2 \int\limits_{\Gamma_{\text{i}} } \psi (y) \big[ \partial_{\nu(y)} \Phi(x,y) + |x|^{2-d} \big]\, \text{d}s_y \quad \text{ for } \, \, x \in \R^d \setminus \Gamma_{\text{i}}.$$ 
Recall that $\Phi(x,z)$ is the fundamental solution to Laplace's equation in $\R^d$ given by 
$$\Phi(x,y)= - \frac{ 1 }{2 \pi } \ln | x-y | \,  \text{ in } \, \R^2  \quad \text{and} \quad  \Phi(x,y)=\frac{1}{4 \pi} \frac{1}{| x-y |}  \, \text{ in } \, \R^3.$$
We refer to \cite{int-equ-book,McLean} for the mapping properties and analysis of the above boundary integral operators.

Since the double layer boundary integral operators satisfy Laplace's equation in $D \setminus \overline{\Omega}$ we make the ansatz that 
\begin{align}
u_0 = (\mathcal{D}_{\text{m}} \, \varphi )(x) + (\widetilde{\mathcal{D}}_{\text{i}} \, \psi )(x) \quad \text{ for } \, \, x\in D \setminus \overline{\Omega}. \label{int-rep}
\end{align}
Using the jump relations for the double layer potentials in \eqref{int-rep} we have that 
\begin{subequations}
\begin{align}
( I- K_{\text{mm}} ) \, \varphi - \widetilde{K}_{\text{im}} \, \psi&= -f \quad \text{ on } \, \, \Gamma_\text{m}  \label{BIE1}\\
K_{\text{mi}} \, \varphi + (I+ \widetilde{K}_{\text{ii}}) \, \psi &= u_0 \big|_{\Gamma_\text{i}} \quad \text{ on } \, \, \Gamma_\text{i}\label{BIE2}
\end{align}
\end{subequations}
where 
$$ K_{\text{pq}} \varphi = (\mathcal{D}_{\text{p}} \varphi )(x) \quad \text{and} \quad \widetilde{K}_{\text{pq}} \psi= (\widetilde{\mathcal{D}}_{\text{p}} \psi )(x) \quad \text{ for } \, \, x \in \Gamma_\text{q}$$ 
with the index $\text{p,q}=$m,i. Notice that we have used that $u_0=f$ on $\Gamma_{\text{m}}$ in equation \eqref{BIE1}. In order to proceed we must show that the system of integral equations in \eqref{BIE1}-\eqref{BIE2} is well-posed. To this end, define the operator 
$$ {\mathcal A} = \begin{bmatrix}  ( I- K_{\text{mm}} )  & - \widetilde{K}_{\text{im}} \\ K_{\text{mi}}   &  (I+ \widetilde{K}_{\text{ii}})  \end{bmatrix} : H^{1/2}(\Gamma_{\text{m}})  \times H^{1/2}(\Gamma_{\text{i}}) \mapsto   H^{1/2}(\Gamma_{\text{m}})  \times H^{1/2}(\Gamma_{\text{i}})$$
which represents the integral operator associated with \eqref{BIE1}-\eqref{BIE2}. 

\begin{theorem}
The operator  ${\mathcal A} : H^{1/2}(\Gamma_{\text{m}})  \times H^{1/2}(\Gamma_{\text{i}}) \mapsto   H^{1/2}(\Gamma_{\text{m}})  \times H^{1/2}(\Gamma_{\text{i}})$ has a bounded inverse.
\end{theorem} 
\begin{proof}
To prove the claim we show that the operator is satisfies the Fredholm alternative and is injective. We begin by proving the injectivity of ${\mathcal A}$. To this end, assume that $(\varphi_1 , \varphi_2)^{\top} \in \text{Null}({\mathcal A})$ which implies that $w(x)=(\mathcal{D}_{\text{m}} \varphi_{\text{1}} )(x) + (\widetilde{\mathcal{D}}_{\text{i}} \varphi_{\text{2}} )(x)$ satisfies Laplace's equation in $D \setminus \overline{\Omega}$ with zero Dirichlet trace on the boundary. The uniqueness of Laplace's equation with zero Dirichlet data implies that $w=0$ in $D \setminus \overline{\Omega}$. The continuity of the normal derivative of the double layer potential on $\Gamma_{\text{m}}$ we have that $w$  satisfies Laplace's equation in $\R^d  \setminus \overline{D}$ with zero Neumann trace on $\Gamma_{\text{m}}$ and uniqueness gives that $w=0$ in $\R^d  \setminus \overline{D}$. Now using the jump relation for the double layer potential we conclude that $\varphi_1 = 0$. This gives that $w(x)=(\widetilde{\mathcal{D}}_{\text{i}} \varphi_{\text{2}} )(x)$ and since $w$ has zero exterior Dirichlet trace on $\Gamma_{\text{i}}$ which implies that $(I+ \widetilde{K}_{\text{ii}})  \varphi_{\text{2}}=0$. Since that operator $(I+ \widetilde{K}_{\text{ii}})$ is injective we have that $\varphi_{\text{2}}=0$, proving the injectivity. 

We now show that ${\mathcal A}$ is the compact perturbation of an invertible operator. To this end, we notice that 
$$ {\mathcal A} = \begin{bmatrix}  ( I- K_{\text{mm}} )  & 0 \\  0   &  (I+ \widetilde{K}_{\text{ii}})  \end{bmatrix} +  \begin{bmatrix}  0 & - \widetilde{K}_{\text{im}} \\ K_{\text{mi}}   &  0  \end{bmatrix}. $$
It is well known (see \cite{int-equ-book}) that both $( I- K_{\text{mm}} )$ and $(I+ \widetilde{K}_{\text{ii}})$ are invertible from $H^{1/2}(\Gamma_{p})$ to itself where $p=$m,i respectively. Next, we show that the operators  
$$K_{\text{mi}}: H^{1/2}(\Gamma_{\text{m}}) \longmapsto  H^{1/2}(\Gamma_{\text{i}}) \quad  \text{ and } \quad   \widetilde{K}_{\text{im}} : H^{1/2}(\Gamma_{\text{i}}) \longmapsto  H^{1/2}(\Gamma_{\text{m}})$$
are compact. Let $v=  ({\mathcal{D}}_{\text{m}} \varphi_{\text{m}} )(x) $ for some $\varphi_{\text{m}}  \in H^{1/2}(\Gamma_{\text{m}})$ which solves Laplace's equation in $D$ and is therefore analytic in the interior of $D$. We can conclude that $v\big|_{\Gamma_\text{i}} =K_{\text{mi}} \varphi_{\text{m}} \in H^{3/2}(\Gamma_{\text{i}})$ and the compactness follows from the compact embedding of $H^{3/2}$ into $H^{1/2}$. A similar argument proves that compactness of the operator $\widetilde{K}_{\text{im}} : H^{1/2}(\Gamma_{\text{i}}) \mapsto  H^{1/2}(\Gamma_{\text{m}})$, which proves the claim since ${\mathcal A} $ is injective and the compact perturbation of an invertible operator. 
\end{proof}

Recall that $u_0 \big|_{\Gamma_\text{i}}$ is still unknown so we use that $\partial_{\nu}  u_0 = g$ on $\Gamma_\text{m}$ to determine the Dirichlet value of the electrostatic potential $u_0$ on $\Gamma_\text{i}$. Solving \eqref{BIE1}-\eqref{BIE2} for $(\varphi ,  \psi )^{\top}$ in terms of $u_0 \big|_{\Gamma_\text{i}}$ we have that \eqref{int-rep} is a representation of $u_0$ in terms of it's Dirichlet data on $\Gamma_\text{i}$. Taking the normal derivative of \eqref{int-rep} on $\Gamma_\text{m}$ gives that 
\begin{align}
g = T_{\text{mm}}\,  \varphi + \widetilde{T}_{\text{im}} \, \psi \quad \text{ for } \, \, x \in \Gamma_\text{m}  \label{eq-from-data}
\end{align}
where the 
%hyper-singular 
operators are given by 
$$ T_{\text{mm}} \, \varphi = \partial_{\nu(x)} (\mathcal{D}_{\text{m}}\,  \varphi )(x)  \quad \text{and } \quad \widetilde{T}_{\text{im}} \, \psi = \partial_{\nu(x)} (\widetilde{\mathcal{D}}_{\text{i}} \, \psi )(x) \quad \text{ for }  x \in \Gamma_\text{m}.$$
%Notice that we have used that $\partial_\nu u_0=g$ on $\Gamma_{\text{m}}$ in equation \eqref{eq-from-data}. 
To recover $u_0 \big|_{\Gamma_\text{i}}$ one solves \eqref{eq-from-data} which can be written as   
\begin{align} \label{data-completion-equation}
{g} = \big[T_{\text{mm}} \quad   \widetilde{T}_{\text{im}} \big] \begin{bmatrix}  ( I- K_{\text{mm}} )  & - \widetilde{K}_{\text{im}} \\ K_{\text{mi}}   &  (I+ \widetilde{K}_{\text{ii}})  \end{bmatrix}^{-1} \begin{bmatrix} -f  \\ u_0 \big|_{\Gamma_\text{i}} \end{bmatrix} \quad \text{ for } \, \, x \in \Gamma_\text{m}. 
\end{align}

Once $u_0 \big|_{\Gamma_\text{i}}$ is known equation \eqref{int-rep} gives that $u_0$ in known for all $x \in D \setminus \overline{\Omega}$ and therefore $\partial_\nu u_0 \big|_{\Gamma_\text{i}}$ is given by taking the normal derivative of \eqref{int-rep} on $\Gamma_\text{i}$. Since the Cauchy data on $\Gamma_\text{i}$ is known the impedance condition $\partial_{\nu}  u_0 +\gamma(x) u_0=0$ can be used to reconstruct the unknown impedance parameter. One can  solve for the impedance 
$$ \gamma(x_n)   = - \frac{ \partial_{\nu} u_0 (x_n)}{u_0 (x_n)}  \quad \text{for} \quad n=1, \cdots ,N \quad \text{with} \, \,\,\,  x_n \in \Gamma_{\text{i}} . $$
One can also consider using a least squares method for recovering the impedance by 
$$\min\limits_{\gamma (x)  \in L^{\infty}(\partial \Omega)} \sum\limits_{n=1}^N \Big| \partial_{\nu}  u_0(x_n) +\gamma(x_n) u_0 (x_n) \Big|^2  \quad \text{ where }\quad  \gamma(x) = \sum\limits_{m=1}^M c_m \Psi_m(x)$$
for some choice of basis functions $\Psi_m$ for $x \in \Gamma_{\text{i}}$.
Since we assume that $\Lambda_0$ is known we can apply this inversion procedure for multiple Cauchy pairs $f_j$ and $g_j =\Lambda_0 f_j$ and determine the impedance parameter $\gamma_j (x)$ for $j=1, \cdots ,M$. Therefore, we can take the reconstructed impedance parameter to be the average of the reconstructions.

%%%%%%%%%%%%%%%%%%%%%%%%%%%%%%%%%%%%%%%%%%%%%%%%%%%%%%%%%%%%%%%%%%%%%%%%
\section{Numerical Validation}\label{numerics}

We now provide some numerical examples of our inversion methods. To do so, we will consider reconstructing both Dirichlet and Impedance inclusions in the unit disk. Recall that, $\phi_z$ is the normal derivative of the greens function in the unit disk with zero trace on the boundary and is therefore given by the Poisson kernel  
$$ \phi_z (\theta) = \frac{1}{2\pi} \frac{1-|z|^2}{|z|^2 +1-2|z| \cos(\theta - \theta_z )}$$
where $\theta_z$ is the polar angle that the point $z$ makes with the positive $x$-axis. 
%We begin by showing that the Theorem \ref{LSM} can be used to reconstruct the inclusion for both the Dirichlet and Impedance boundary condition. Once the inclusion is reconstructed by the sampling method we then turn to giving numerical reconstructions of the impedance parameter. 

%%%%%%%%%%%%%%%%%%%%%%%%%%%%%%%%%%%%%%%%%%%%%%%%%%%%%%%%%%%%%%%%%%%%%%%%
\subsection{Reconstruction of a Dirichlet inclusion }
For this case we only consider a simple example and will give more substantial reconstructions for the case of an impedance condition. 
Assume that the boundary of the inclusion $\partial \Omega=\rho \big( \cos(\theta) , \sin(\theta) \,  \big)$ where $0 < \rho<1$. Since we assume that $D$ is the unit disk in $\R^2$ we attempt to find a representation of the electrostatic potential $u_0(r ,\theta)$ which solves the problem 
\begin{align*}
\Delta u_0(r, \theta)=0 \quad \text{for all} \quad \rho < r<1 \quad \text{ and } \quad   \theta \in [0 , 2 \pi)\\  
u_0(1, \theta) = f(\theta) \quad \text{and} \quad u_0(\rho , \theta)=0.
\end{align*}
Now, since $u_0(r , \theta)$ solves Laplace's equation in an annular region 
we assume it can be written as linear combination of solutions to the to problem in the annuals and therefore has the form
\begin{align*}
u_0(r,\theta)=a_0 +b_0 \ln r +  \sum_{|n|=1}^{\infty} \big(a_n r^{|n|} + b_n r^{-|n|} \big)  \, \text{e}^{ \text{i} n \theta}. 
\end{align*}
By applying the boundary conditions we have that (see \cite{iterative-inclusion})
\begin{align*}
u_0(r,\theta)=\frac{f_0}{\ln \rho} \ln \left( \frac{\rho}{r} \right) + \sum_{|n|=1}^{\infty}  \frac{f_n}{1-\rho^{2|n|}}  \left(r^{|n|} - r^{-|n|}\rho^{2|n|} \right)  \text{e}^{ \text{i} n \theta}
\end{align*}
where $f_n$ are the Fourier coefficients for $f$ given by 
$$ f_n= \frac{1}{2 \pi} \int\limits_{0}^{2 \pi} f(\phi)\,  \text{e}^{- \text{i} n \phi} \, d \phi .$$
Therefore, by taking the derivative with respect to $r$ gives that 
$$ \partial_r u_0(1,\theta)= - \frac{f_0}{\ln \rho }+ \sum_{|n|=1}^{\infty}  |n|  f_n \, \frac{1+\rho^{2|n|}}{1-\rho^{2|n|}} \,  \text{e}^{ \text{i} n \theta}.$$
It is clear that the electrostatic potential for a material without a perfectly conducting inclusion is given by 
\begin{align}
u(r,\theta)={f_0} + \sum_{|n|=1}^{\infty}  {f_n} r^{|n|} \text{e}^{ \text{i} n \theta} \quad \text{and} \quad  \partial_r u(1,\theta)= \sum_{|n|=1}^{\infty} |n|  f_n  \text{e}^{ \text{i} n \theta}. \label{potential}
\end{align}
This now gives that the difference of the DtN mapping is given by 
\begin{align}
(\Lambda - \Lambda_0 ) f =\frac{f_0}{\ln \rho} -2 \sum_{|n|=1}^{\infty}  |n|  \frac{ \rho^{2|n|}}{1-\rho^{2|n|}} f_n \, \text{e}^{ \text{i} n \theta}. \label{dtn-series}
\end{align}
By interchanging summation and integration we obtain that 
$$(\Lambda - \Lambda_0 ) f =   \int\limits_{0}^{2 \pi} K(\theta , \phi) f(\phi) \, d \phi $$
where the kernel is given by 
\begin{align}
K(\theta , \phi) = \frac{1}{ 2 \pi \ln \rho} -\frac{1}{\pi} \sum_{|n|=1}^{\infty}  |n| \frac{ \rho^{2|n|}}{1-\rho^{2|n|}}   \text{e}^{ \text{i} n( \theta - \phi)}. \label{def-k}
\end{align}

We now consider the approximation of $(\Lambda - \Lambda_0)$ by a truncated series. In our experiments we will take the terms for $0 \leq |n| \leq 20$. In the following we see that the converges of the truncated series $(\Lambda - \Lambda_0)_N$ converges exponentially fast to $(\Lambda - \Lambda_0)$ as $N \rightarrow \infty$.

{
\begin{theorem}\label{DtN-approx}
Let $(\Lambda - \Lambda_0)_N : H^{1/2}(0,2 \pi) \longmapsto H^{-1/2}(0,2 \pi)$ be the truncated series approximation of $(\Lambda - \Lambda_0)$ given by \eqref{dtn-series} then we have that 
$$\big\| (\Lambda - \Lambda_0) -(\Lambda - \Lambda_0)_N \big\| \leq C \rho^{2(N+1)} $$
where $\| \cdot \|$ is the operator norm on $\mathcal{L} \big( H^{1/2}(0,2 \pi) \, , \, H^{-1/2}(0,2 \pi) \big).$
\end{theorem}
\begin{proof}
To begin, let $f \in H^{1/2}(0,2 \pi)$ then we have that by \eqref{dtn-series}
\begin{align*}
(\Lambda - \Lambda_0)f -(\Lambda - \Lambda_0)_N f  = -2  \sum_{|n|=N+1}^{\infty}  |n|  \frac{ \rho^{2|n|}}{1-\rho^{2|n|}} f_n \, \text{e}^{ \text{i} n \theta}.
\end{align*}
Now by the Cauchy-Schwarz inequality in $\ell^2$ we have that 
\begin{align*}
\Big|  (\Lambda - \Lambda_0)f -(\Lambda - \Lambda_0)_N f \Big|^2  &\leq  C  \sum_{|n|=N+1}^{\infty}  |n|  \frac{ \rho^{4|n|}}{(1-\rho^{2|n|})^2}  \big| \text{e}^{ \text{i} n \theta} \big|^2 \, \cdot \sum_{|n|=N+1}^{\infty}  |n| |f_n|^2 \\
												       &\leq C \big\| f \big\|^2_{H^{1/2}(0,2 \pi)} \, \sum_{|n|=N+1}^{\infty}  |n|  \rho^{4|n|}. 
\end{align*}
We have used that 
$$ \frac{1}{2 \pi} \big\| f \big\|^2_{H^{1/2}(0,2 \pi)} = \sum_{|n|=0}^{\infty} \big(1+ |n|^2\, \big)^{1/2}\,  |f_n|^2.$$
Now, notice that 
$$\Big|  (\Lambda - \Lambda_0)f -(\Lambda - \Lambda_0)_N f \Big|^2 \leq C\, \rho^{4(N+1)}\, \big\| f \big\|^2_{H^{1/2}(0,2 \pi)}  .$$
Since the $H^{-1/2}(0,2 \pi)$-norm is bounded by $L^{\infty}(0,2 \pi)$-norm we can conclude that 
$$ \big\| (\Lambda - \Lambda_0) -(\Lambda - \Lambda_0)_N \big\| \leq C\rho^{2(N+1)}$$
proving the claim. 
\end{proof}

%%% end color 
}

We can approximate the difference of the DtN mappings $(\Lambda - \Lambda_0)$ where we apply Theorem \ref{LSM} to the discretized operator. We discretize our operator by using a simple collocation method with $64$ equally spaced points in the interval $[0 , 2\pi )$. This gives a $64 \times 64$ matrix approximation of $(\Lambda - \Lambda_0)$ which we will denote ${\bf A}$ and a vector $ {\bf b}_z=[\phi_z(\theta_j)]_{j=1}^{64}$ where $\theta_j$ are the collocation points.  In our calculations we add random noise to the DtN mappings given by ${\bf A}^{\delta}_{i,j}={\bf A}_{i,j}\big( 1 +\delta {\bf E}_{i,j} \big)$ where the mean zero random matrix ${\bf E}$ {satisfying } $\| {\bf E} \|_2 =1$. This gives a discretized version of \eqref{PGE}  
\begin{align}
{\bf A}^{\delta} {\bf f}_z =  {\bf b}_z \quad \text{ for } \quad z \in D. \label{matrix-equ}
\end{align}
Since the operator $(\Lambda - \Lambda_0)$ is compact we have that it's matrix approximation is ill-conditioned. In order to solve \eqref{matrix-equ} we use  Tikhonov's regularization. To this end, we let $\sigma_i \in \R^+$ be the singular values with $ {\bf u}_i$ and $ {\bf v}_i$ in $\C^{64}$ the singular vectors of the matrix ${\bf A}^{\delta}$. We let $ {\bf f}_z^{\text{Tik}}$ be the regularized solution to \eqref{matrix-equ} given by 
\begin{align*} 
 {\bf f}_z^{\text{Tik}} = \sum\limits_{i=1}^{64} \frac{\sigma_i}{\alpha(\delta) + \sigma^2_i} \, ( {\bf b}_z ,  {\bf v}_i)_{\ell^2} \,  {\bf u}_i. 
 \end{align*}
Here $ {\bf f}_z^{\text{Tik}}$ denotes the Tikhonov's regularization solution where $\alpha$ is chosen by the discrepancy principle. Therefore, to reconstruct the inclusions we define the function 
$$W(z) =\big\|  {\bf f}_z^{\text{Tik}} \big\|^{-1}_{\ell^2}. $$
Even though we plot the weaker $L^2$-norm we see that this is sufficient to approximate the inclusion $\Omega$. In the following experiments we take the uniformly distributed noise level $\delta=0.05$ where we plot the indicator function $W(z)$, see Figures \ref{reconstruct1} and \ref{reconstruct2}.

\begin{figure}[ht!]
\hspace{-0.8in}\includegraphics[scale=0.25]{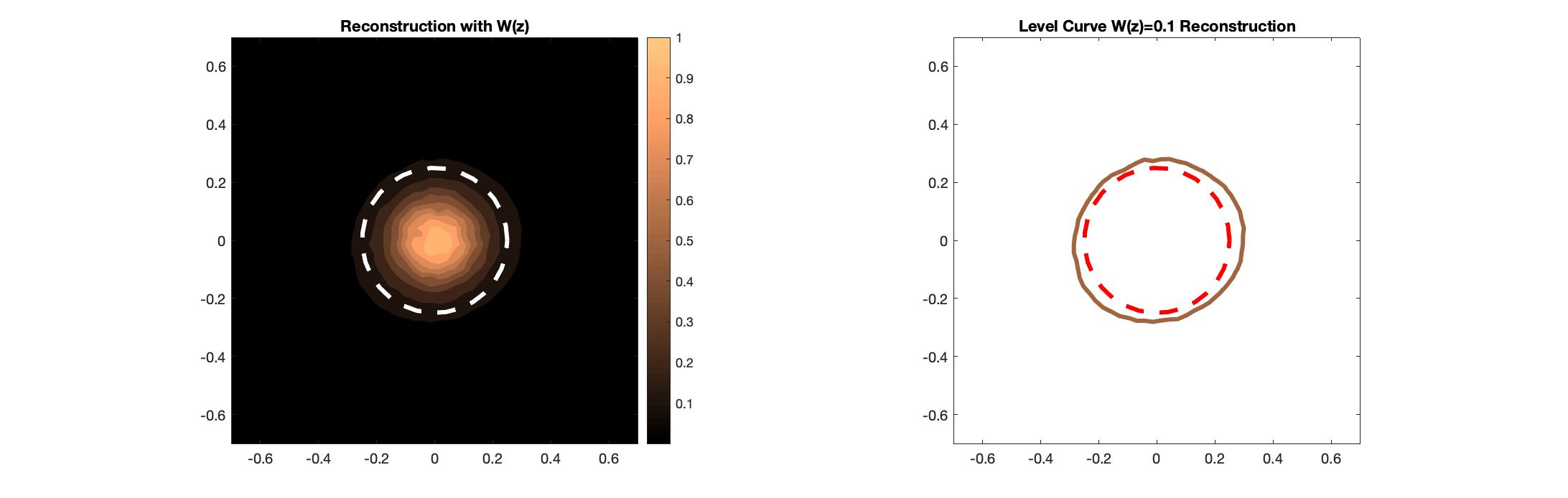} 
\caption{Reconstruction of a perfectly conducting circular inclusion by the Sampling Method with radius $\rho=0.25$ }
\label{reconstruct1}
\end{figure}

\begin{figure}[ht!]
\hspace{-0.8in}\includegraphics[scale=0.25]{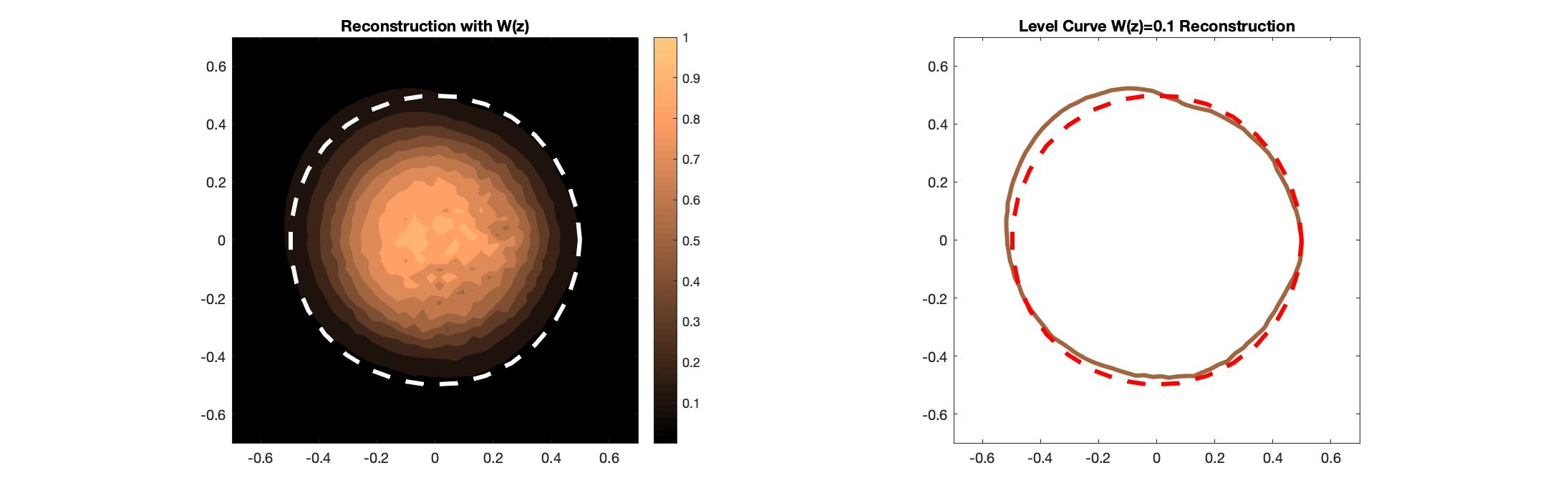}
\caption{Reconstruction of a perfectly conducting circular inclusion by the Sampling Method with radius $\rho=0.5$ }
\label{reconstruct2}
\end{figure}

%%%%%%%%%%%%%%%%%%%%%%%%%%%%%%%%%%%%%%%%%%%%%%%%%%%%%%%%%%%%%%%%%%%%%%%%
\subsection{Reconstruction of an impedance inclusion}
To reconstruct an inclusion with an impedance coefficient $\gamma\big(x(\theta)\big)$ we us a boundary integral equation to simulate the DtN mappings. To this end, we assume that $u_0$ can be written as a combination of a double layer potential on $\partial D=\Gamma_{\text{m}}$ and a single layer potential on $\partial \Omega = \Gamma_{\text{i}}$. Here the boundary of $\Gamma_{\text{m}}$ is given by the boundary of the unit disk and $\Gamma_{\text{i}}$ is given by $x(\theta): [0,2 \pi] \mapsto \R^2$  which is a $2\pi$-periodic representation of the $C^2$ boundary. Applying the boundary conditions 
$$ u_0 \big|_{\Gamma_{\text{m}}}=f  \quad \text{ and } \quad \big( \partial_{\nu} u_0 + \gamma  u_0 \big) \big|_{\Gamma_{\text{i}}}=0$$
gives a $2 \times 2$ system of boundary integral equations. The boundary integral equations are solved via the Nystr\"{o}m  method using 32 equally spaced points which gives a representation of $u_0$ in $D \setminus \overline{\Omega}$. This should give a sufficiently accurate approximation for the electrostatic potential due to the exponential convergence (see \cite{int-equ-book}). We then compute  $\Lambda_0\,  \text{e}^{\text{i} n \theta}$ where $\Lambda_0$ is the DtN mapping for the material with the inclusion by taking the normal derivative of $u_0$ on $\Gamma_{\text{m}}$. It is clear that $\Lambda \, \text{e}^{\text{i} n \theta} = n \text{e}^{\text{i} n \theta}$ for all $n \in \Z$. To obtain a discretized version of \eqref{PGE} we consider 
$$f_z \approx \sum\limits_{n=0}^{19} f^z_n \text{e}^{\text{i} n \theta} \quad \text{which implies that } \quad \sum\limits_{n=0}^{19} f^z_n (\Lambda-\Lambda_0) \text{e}^{\text{i} n \theta} \approx \phi_z $$
and solve for the first 20 Fourier coefficients to the solution $f_z$ to \eqref{PGE}. In our experiments we solve the above equation for $\theta_j \in [0 , 2\pi )$ where $\theta_j$ are taken to be 20 equally spaced points. This gives a $20 \times 20$ linear system which is solved using a spectral cut-off, where the cut-off parameter is chosen to be $10^{-4}$. To visualize the inclusion as in the previous examples we let 
$$W(z)=\left[ \sum\limits_{n=0}^{19} \big| f^z_n \big|^2 \right]^{-1/2} . $$
We implement this for three different inclusions given by 
\begin{align*}
\text{Circular shaped inclusion:} \quad x(\theta) &= \big( 0.3 \cos(\theta) \, , \, 0.3 \sin(\theta) \big)\\
\text{Elliptical shaped inclusion:} \quad x(\theta) &= \big( 0.5 \cos(\theta) \, , \, 0.3 \sin(\theta) \big)\\
\text{Cardioid shaped inclusion:} \quad x(\theta) &= \frac{0.35+0.3\cos(\theta)+0.05\sin(2\theta)}{1+0.7\cos(\theta)} \, \big( \cos(\theta) \, , \,  \sin(\theta) \big)
\end{align*}
with the impedance parameter $\gamma \big(x(\theta)\big) = 2-\sin^4(\theta)$ where $4\%$ mean zero uniformly distributed random noise is added to the flux data measurements, see Figures \ref{recon7}, \ref{recon8} and \ref{recon9}.  

\begin{figure}[H]
\includegraphics[scale=0.18]{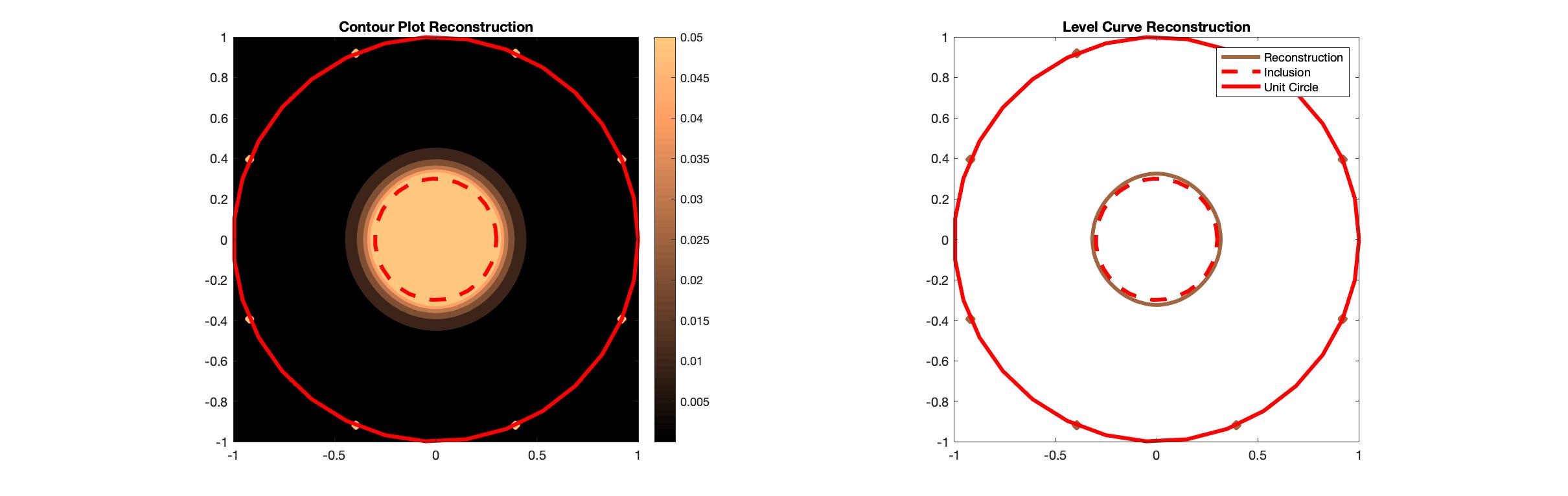}
\caption{Reconstruction the circle via the Sampling Method with impedance parameter $\gamma\big(x(\theta)\big)=2-\sin^4(\theta)$ with cut-off parameter $10^{-4}$. }
\label{recon7}
\end{figure}

\begin{figure}[H]
\includegraphics[scale=0.18]{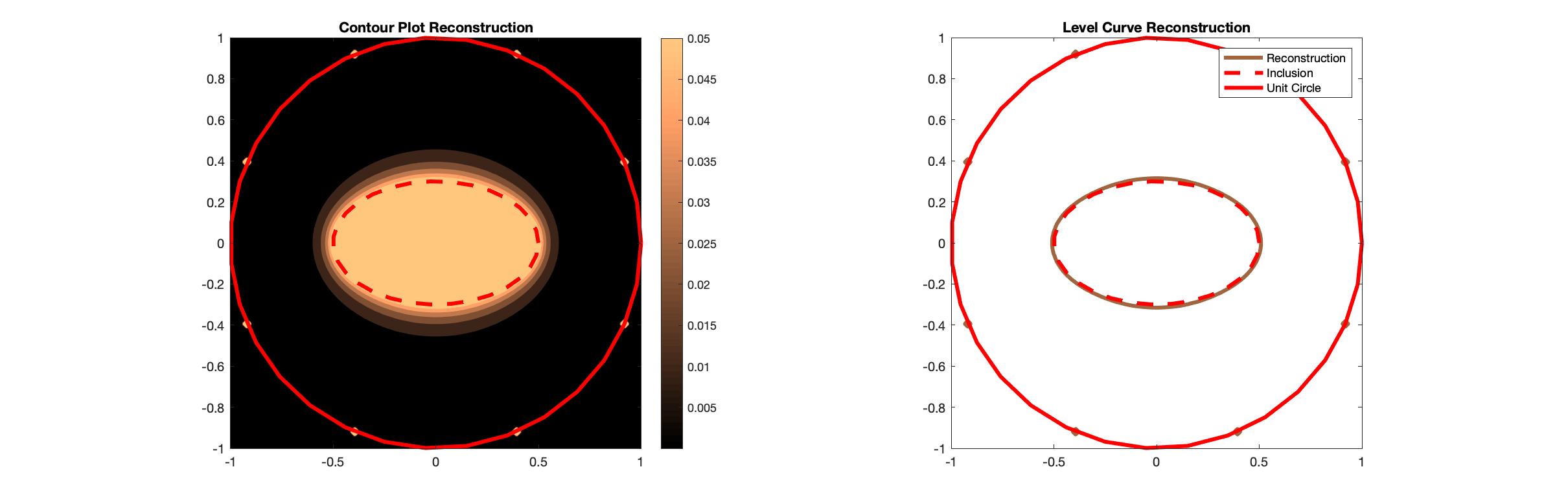}
\caption{Reconstruction of the ellipse via the Sampling Method with impedance parameter $\gamma\big(x(\theta)\big)=2-\sin^4(\theta)$ with cut-off parameter $10^{-4}$. }
\label{recon8}
\end{figure}

\begin{figure}[H]
\includegraphics[scale=0.18]{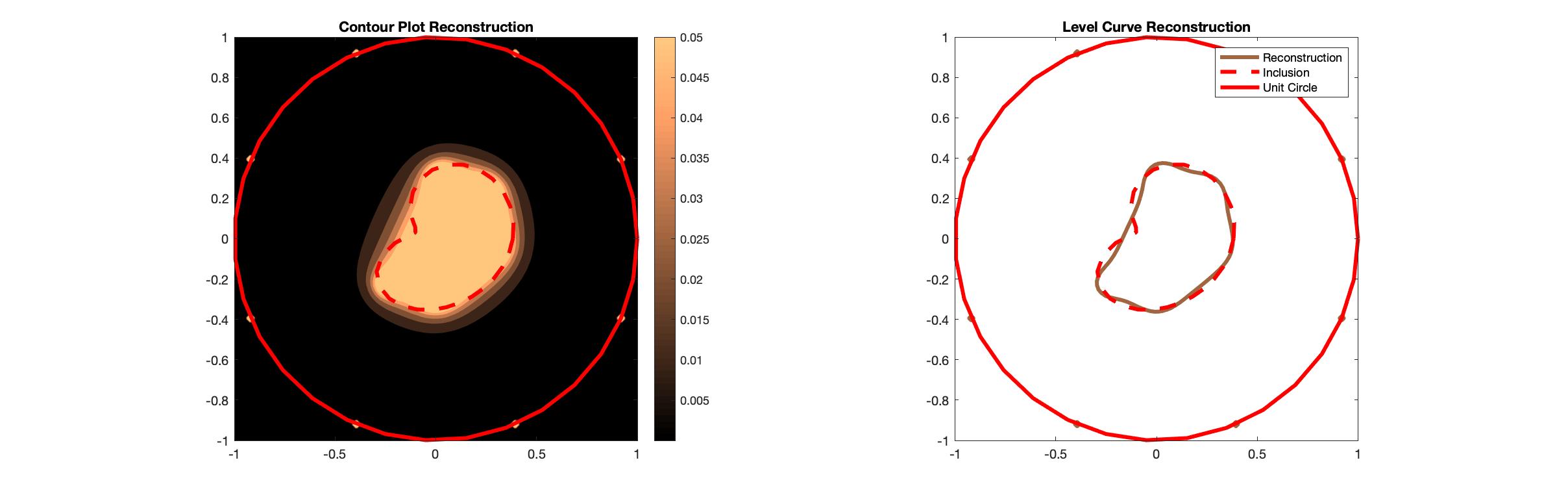}
\caption{Reconstruction of the cardioid via the Sampling Method with impedance parameter $\gamma\big(x(\theta)\big)=2-\sin^4(\theta)$ with cut-off parameter $10^{-4}$.}
\label{recon9}
\end{figure}

%%%%%%%%%%%%%%%%%%%%%%%%%%%%%%%%%%%%%%%%%%%%%%%%%%%%%%%%%%%%%%%%%%%%%%%%
\subsection{Reconstruction of the impedance parameter}

We now give a numerical example of recovering the impedance parameter using the method described in Section \ref{BIE}. Therefore, we present an example where the boundary has been reconstructed by Theorem \ref{LSM}. Here we consider the ellipse $x(\theta) = \big( 0.5 \cos(\theta) \, , \, 0.3 \sin(\theta) \big)$ with impedance parameter $\gamma\big(x(\theta)\big)=2-\sin^4(\theta)$ from the previous section. In our calculations we first  represent the reconstructed curve using trigonometric polynomials. To this end, we assume that the inclusion $\Omega$ is centered at the origin and taking the values on the level curve given in Figure \ref{recon7} we approximate 
$$ x(\theta) = \big( x_1(\theta) , x_2(\theta) \big) \quad \text{ such that } \quad x_p (\theta) = \sum\limits_{m=1}^{M} a_m^{(p)} \cos (m \theta) +b_m^{(p)} \sin (m \theta) $$
where $p=1,2$. The coefficients $a_m^{(p)}$ and $b_m^{(p)}$ are solved for in the least squares sense with Tikhonov regularization such that $x(\theta)$ approximates the reconstructed curve. In our calculations we penalize the $H^2(0,2 \pi)$ norm of $x_p (\theta)$ by taking the regularization parameter based on the level of noise in the data. 

Now that we have an approximation of $x(\theta)$ we can reconstruct the impedance using boundary integral equations. We apply the data completion algorithm described in Section \ref{BIE} to recover the Cauchy data on the interior boundary $\Gamma_{\text{i}}$. Using the same method as in the previous Section for any given $f$ we can compute the corresponding $\Lambda_0 f$.
Note that our original data on $\Gamma_{\text{m}}$ is subject to $4\%$ mean zero random noise and these errors transfer to the reconstruction of $\Gamma_{\text{i}}$. This gives that to reconstruct $\gamma\big(x(\theta)\big)$ we must solve a discretized version of \eqref{data-completion-equation} where the Nystr\"{o}m  method using 64 points is used to discretize the equation. Using a standard Tikhonov regularization scheme we solve the discretized version of \eqref{data-completion-equation} which allows use to determine $u_0$ and $\partial_{\nu} u_0$ on $\Gamma_{\text{i}}$ for a given $f$. In our calculations we take $f (\theta) = \cos (k \theta) \text{ and } \sin (k \theta)$ for $k=1, \cdots ,8$  which corresponds to having 16 voltage and current measurements. For each $f$ the impedance is computed by 
$$  \gamma \big(x(\theta_j) \big )   = - \frac{ \partial_{\nu} u_0 \big(x(\theta_j) \big )}{u_0 \big(x(\theta_j) \big )} \quad \text{where} \quad \theta_j = \frac{2 j \pi}{64}  \quad  \text{ for }\quad j = 0,\cdots, 64.$$ 
In Figure \ref{recon-imped} we show the approximation of the reconstructed ellipse as well as  the plot of the reconstructed impedance which is obtain by averaging the 16 results. 

\begin{figure}[H]
\includegraphics[scale=0.35]{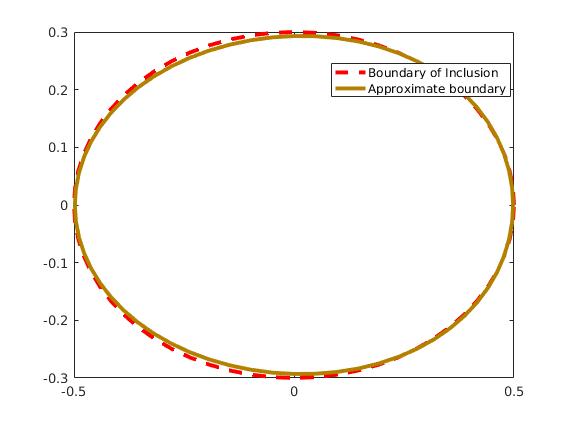}\includegraphics[scale=0.35]{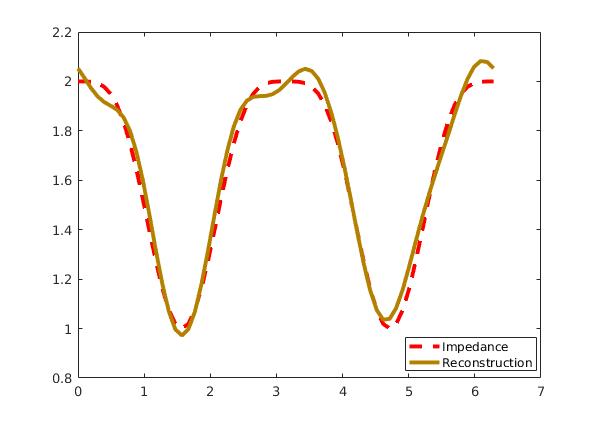}
\caption{On the left an approximation of the boundary of the inclusion for $M=7$. On the right is the reconstruction of the impedance  $\gamma\big(x(\theta)\big)=2-\sin^4(\theta)$ from 16-Cauchy pairs.}
\label{recon-imped}
\end{figure}

%%%%%%%%%%%%%%%%%%%%%%%%%%%%%%%%%%%%%%%%%%%%%%%%%%%%%%%%%%%%%%%%%%%%%%%%

\section{Conclusion}
In conclusion, we have developed the linear sampling method for recovering two types of impenetrable inclusion from electrostatic data. Also, we have derived a direct method to recover that boundary condition on the recovered impenetrable inclusion without a-prior knowledge of the boundary condition.  An alternative sampling method is the factorization method (see \cite{kirschbook}) where one proves that the range of a known operator defined by the measurements operator uniquely determines the region $\Omega$ and gives a simple numerical inversion algorithm. In \cite{EIT-inclusion} the factorization method has been used to reconstruct penetrable inclusions from electrostatic Cauchy data. To our knowledge, the only manuscript that analyzes the factorization method for the EIT problem with an impenetrable inclusion is done in \cite{gibc-eit} where the coefficients in the boundary condition are assumed to complex valued. Applying the  factorization method for the problems considered in this paper is on going research. In \cite{MUSICImp} the MUSIC algorithm, which can be seen as a discrete version of the factorization method, was derived to detect corrosion of a known interior boundary. For many inverse boundary value problems for elliptic equations the factorization method has been used to validate the linear sampling method using the eigenvalue decomposition of the measurements see \cite{arens,LSM-AL}.

%%%%%%%%%%%%%%%%%%%%%%%%%%%%%%%%%%%%%%%%%%%%%%%%%%%%%%%%%%%%%%%%%%%%%%%%

\section{Acknowledgments}
The research was started while the author was at Texas A$\&$M University. The university's hospitality during the time the author was there is greatly appreciated. The author would also like to thank William Rundell for providing the code to solve the direct problem and for helpful discussions during the authors time at Texas A$\&$M University and beyond.

%%%%%%%%%%%%%%%%%%%%%%%%%%%%%%%%%%%%%%%%%%%%%%%%%%%%%%%%%%%%%%%%%


\begin{thebibliography}{99}

%\bibitem{detect-inclusion}
%G. Alessandrini, A. Morassi and E. Rosset, Detecting cavities by electrostatic boundary measurements, {\it Inverse Problems}, {\bf 18}, (2002) 1333-1353

\bibitem{MUSICImp} H. Ammari et al.,  {\newblock A MUSIC-type algorithm for detecting internal corrosion from electrostatic boundary data}, {\it Numer. Math.} {\bf 108} (2008), 501-5028.


\bibitem{arens} T. Arens,  {\newblock Why linear sampling method works}, {\it Inverse Problems} {\bf 20} (2004), 163-173.


\bibitem{LSM-AL}
T. Arens and A. Lechleiter, Indicator Functions for Shape Reconstruction Related to the Linear Sampling Method,
{\it SIAM Journal on Imaging Sciences}, {\bf 8(1)}, (2015),  513-535.


\bibitem{unique-imp}
V. Bacchelli, Uniqueness for the determination of unknown boundary and impedance with the homogeneous Robin condition, {\it Inverse Problems} {\bf 25}, (2009) 015004


\bibitem{C-Map-imp} 
F. Ben Hassen, Y. Boukari and H. Haddar, Inverse impedance boundary problem via the
conformal mapping method: the case of small impedances {\it Revue ARIMA}, {\bf 13} (2010), pp. 47-62


\bibitem{EIT-inclusion}
M. Bruhl, Explicit characterization of inclusions in electrical impedance tomography, {\it SIAM J. Math. Anal.}, {\bf 32}, No. 6, (2001) 1327--1341


\bibitem{CCbook} F. Cakoni  and  D. Colton, \emph{ ``A Qualitative Approach to Inverse Scattering Theory''}, Springer, Berlin 2014.


\bibitem{lsmaniso} 
F. Cakoni, D. Colton and H. Haddar, The linear sampling method for anisotropic media,  {\it J.Comput. Appl. Math.}, {\bf 146}, (2002), 285-299.
 
\bibitem{CK-gibc}
F. Cakoni , Y. Hu and R. Kress, Simultaneous reconstruction of shape and generalized impedance functions in electrostatic imaging, {\it Inverse Problems} {\bf 30} (2014) 105009.


\bibitem{CK} D. Colton and A. Kirsch, {\newblock A simple method for solving inverse scattering problems in the resonance region}, {\it Inverse Problems} {\bf 12} (1996), 383-393.

\bibitem{coltonkress}
D.~Colton and R.~Kress,
\newblock {\em ``Inverse Acoustic and Electromagnetic Scattering Theory''},
\newblock Springer, New York, third edition, 2013.


\bibitem{iterative-inclusion}
\newblock H. Haddar and R. Kress, {Conformal mappings and inverse boundary value problems,}  {\it Inverse Problems }, {\bf 21}, (2005), 935-953.

\bibitem{TDLSM-wave}
H. Haddar, A. Lechleiter and S. Marmorat, An improved time domain linear sampling method for Robin and Neumann obstacles, {\it Applicable Analysis}, {\bf 93} (2014) 369-390. 


\bibitem{gibc-eit} %sampling method for this problem in electrostatics 
\newblock I. Harris, 
\newblock Detecting an inclusion with a generalized impedance condition from electrostatic data via sampling.
\newblock {\it Mathematical Methods in the Applied Sciences}, DOI: 10.1002/mma.5777 

\bibitem{evans}
 L. Evans, \emph{``Partial Differential Equations"},  2$^{nd}$ edition, AMS 2010.

\bibitem{TDLSM-elastic}
N. Khaji and  S.H. Dehghan Manshadi, Time domain linear sampling method for qualitative identification of buried cavities from elastodynamic over-determined boundary data. {\it Computers $\&$ Structures} {\bf 153}, (2015), 36-48.

 \bibitem{kirschbook}
A. Kirsch and N.  Grinberg, \emph{The Factorization Method for Inverse Problems}. Oxford University Press, Oxford 2008.

 \bibitem{FM-EIT}
A. Kirsch, The Factorization Method for a Class of Inverse Elliptic Problems. {\it Math. Nachrichten}  {\bf 278} (2005), 258-277.

 \bibitem{uniqueness}
R. Kress, Inverse Dirichlet problem and conformal mapping. {\it Math. Comput. Simul.}, {\bf 66} (2004) 255-265.

\bibitem{int-equ-book}
R. Kress, {\it ``Linear Integral Equations''}, Springer, New York, third edition, 2014. 

\bibitem{McLean}
W. McLean,{\it``Strongly elliptic systems and boundary integral equations''}. Cambridge: Cambridge University Press 2000.

\bibitem{impedance-heat}
G. Nakamura and H. Wang, Numerical reconstruction of unknown Robin inclusions inside a heat conductor by a non-iterative method, {\it Inverse Problems} {\bf 33} (2017), 055002

\bibitem{EIT-impedance}
W. Rundell, Recovering an obstacle and its impedance from Cauchy data, {\it Inverse Problems} {\bf 24} (2008), 045003



\end{thebibliography}
\end{document}